\documentclass{amsproc}

\usepackage{amsmath}
\usepackage{amssymb}
\usepackage{graphicx}
\DeclareGraphicsExtensions{.png,.pdf,.eps}
\usepackage{amsthm}
\usepackage[all,2cell]{xy}

\newtheorem{thm}{Theorem}[section]
\newtheorem{lem}[thm]{Lemma}
\newtheorem{cor}[thm]{Corollary}

\def\G{{\mathbb G}}

\def\P{{\mathbb P}}
\def\Q{{\mathbb Q}}

\def\Z{{\mathbb Z}}

\def\cL{{\mathcal L}}

\def\cO{{\mathcal{O}}}
\def\cQ{{\mathcal{Q}}}

\def\Q{{\mathbb{Q}}}
\def\G{{\mathbb{G}}}

\def\operatorname#1{\mathop{\rm #1}\nolimits}

\def\Proj{\operatorname{Proj}}

\def\Pic{\operatorname{Pic}}

\def\qed{\hspace{\fill}$\rule{2mm}{2mm}$}

\begin{document}

\title{Rank two Fano bundles on $\G(1,4)$}

\author{Roberto Mu\~noz}
\address{Departamento de Matem\'atica Aplicada, ESCET, Universidad
Rey Juan Carlos, 28933-M\'ostoles, Madrid, Spain}
\email{roberto.munoz@urjc.es}
\thanks{First and third authors partially supported by the Spanish government
project
  MTM2009-06964. Second author partially supported by  MIUR (PRIN project:
Propriet\`a
geometriche delle variet\`a reali e complesse)}
\author{Gianluca Occhetta} \address{Dipartimento di Matematica,
  Universit\`a di Trento, Via Sommarive 14,
I-38123, Povo (Trento), Italy}
\email{gianluca.occhetta@unitn.it}

\author{Luis E. Sol\'a Conde}
\address{Departamento de Matem\'atica Aplicada, ESCET, Universidad
Rey Juan Carlos, 28933-M\'ostoles, Madrid, Spain}
\email{luis.sola@urjc.es}



\begin{abstract} We classify rank two Fano bundles over the Grassmannian of lines $\G(1,4)$. In particular
we show that the only non-split rank two Fano bundle  over  $\G(1,4)$ is, up to a twist, the universal quotient bundle $\cQ$. This completes the classification of rank two Fano bundles over Grassmannians of lines.
\end{abstract}

\maketitle

\section{Introduction}\label{sec:intro}

The problem of classifying low rank vector bundles on grassmannians appears naturally in the framework of Hartshorne's Conjecture. 

On one hand one may consider finite morphisms from the grassmannian to the projective space and take pull-backs of vector bundles via these morphisms. Simple computations on Chern classes will discriminate which vector bundles on the grassmannian can appear as a pull-back of one on a projective space, relating both classifications: that of low rank vector bundles on the grassmannian and that of low rank vector bundles on the corresponding projective space. Let us focus on the codimension two case where Hartshorne's conjecture can be stated as follows: any rank two vector bundle on $\P^6$ decomposes as the direct sum of two line bundles. By the previous ideas this conjecture would follow from the fact that any rank two vector bundle on the grassmannian of lines in $\P^4$, $\G(1,4)$, either decomposes as a sum of line bundles or is, up to a twist, isomorphic to the universal quotient bundle $\cQ$.
This path has been followed in a number of papers, see for instance \cite{AG}, \cite{AM}, \cite{M1}, \cite{M3} and \cite{O}, where the authors study extensions to $\G(1,4)$ of Horrocks decomposability criterion, that had been previously shown to work on projective spaces (\cite{Ho}) and quadrics (\cite{M2}). In the case of $\G(1,4)$ these results show essentially how the vanishing of certain cohomology groups characterizes decomposable bundles and twists of the universal quotient bundle $\cQ$. Note that a straightforward computation of Chern classes shows that no bundle on $\P^6$ may be pulled-back to $\cQ$ or their twists.

On the other hand one may consider the problem of classifying low rank vector bundles on other Fano manifolds of Picard number one and, in particular, on grassmannians, as a natural extension of the decomposability question on vector bundles on the projective space. In this setting, it is known that some partial results may be achieved under certain positivity conditions:  for instance, rank two {\it Fano bundles} over projective spaces and quadrics are completely classified (see \cite{SW1,SW,SSW,APW}). Furthermore, it has been noted by Malaspina in \cite
{M2} that this classification provides precisely the complete list of rank two bundles on projective spaces and quadrics satisfying the decomposability criteria that we have referred to above.


In our recent paper \cite{MOS} we classified rank two Fano bundles over  $\G(1,n)$ with $n \ge 5$, proving that they are twists of the universal quotient bundle $\cQ$ or sums of line bundles (\cite[Corollary 5.17]{MOS}). Our proof relied on showing that the restriction of a Fano bundle $E$ on $\G(1,n)$ to a $\P^{n-1} \subset \G(1,n)$,
representing lines through a fixed point, is a sum of line bundles; this allows us to conclude by using a classification of {\it uniform} (i.e, whose restriction to every line is the same) vector bundles on grassmannians (see \cite[Th\'eor\`eme 1]{G} or \cite[Theorem 4.1]{MOS}). Note that the restriction of $E$ to a $\P^{n-1}$ is not necessarily Fano; however it yet verifies a weaker positivity condition, that we call
$1$-Fano (see \cite[Definition 5.1]{MOS}), from which we infer the splitting (\cite[Theorem 5.15]{MOS}) for $n \geq 5$. Unfortunately there are well known examples of indecomposable $1$-Fano bundles on $\P^3$, including the null-correlation bundle ($c_1=0$, $c_2=1$) and the stable bundles with $c_1=0$, $c_2=2$ (\cite[Rem.~9.4.1]{Ha}). This prevents our arguments for $\G(1,n)$ from working in the case $n=4$. Note also that the cases $n=2,3$ follow from the classification of Fano bundles on projective spaces and quadrics, so that, at this point, the only grassmannian of lines that could eventually support a different Fano bundle was $\G(1,4)$. In this note we prove the following:

\begin{thm}\label{thm:main}
Let $E$ be a rank two Fano bundle on $\G(1,4)$; then $E$ is either a twist of the universal quotient bundle or a direct sum of two line bundles.
\end{thm}

Note that a vector bundle $\cO(a)\oplus\cO(b)$ on $\G(1,n)$ is Fano if and only if $|a-b|< n+1$, hence, up to a twist with a line bundle, the list of Fano bundles on $\G(1,n)$ is finite for all $n$. In the case $n=4$, split Fano bundles are twists of one of the following:
$$
\cO^{ \oplus 2},\,\cO(-1) \oplus \cO(1),\,\cO(-2) \oplus \cO(2),\,\cO(-1) \oplus \cO,\,\cO(-2) \oplus \cO(1).
$$

Our proof  of Theorem \ref{thm:main} does not involve a classification of $1$-Fano bundles on $\P^3$, which to our best knowledge is still unknown. We rather consider the restriction of $E$ to additional subvarieties in different cohomology classes and use the techniques of \cite{MOS} (positivity of Schur polynomials, Schwarzenberger conditions, Riemann-Roch combined with vanishing theorems) to compute a manageable list of possible Chern classes of $E|_{\P^3}$ (with the help of the Maple package Schubert \cite{S}). At this point a case by case analysis of $E|_{\P^3}$ finishes the proof.

\subsection{Notation}\label{ssec:not}
Along this paper $\G(1,4)$ will denote the Grassmann variety parametrizing lines in the complex projective space of dimension $4$, and we will consider vector bundles $E$ of rank two on $\G(1,4)$. Given an integer $j$, we will denote by $E(j)$ the twist of $E$ with the $j$-th tensor power of the ample generator of $\Pic(\G(1,4))$.

Given integers $i,j$ such that $0\leq i<j\leq 4$, we will denote by $\Omega(i,j)$ the cohomology class of the subscheme of $\G(1,4)$ parameterizing lines contained in a linear subspace $P_j\subset\P^4$ of dimension $j$ and meeting a linear subspace $P_i\subset P_j$ of dimension $i$. Since $H^2(\G(1,4),\Z) \simeq \mathbb Z\langle\Omega(2,4) \rangle$  and  $H^4(\G(1,4), \mathbb Z)= \mathbb Z \langle \Omega(1,4) \rangle \oplus \mathbb Z\langle  \Omega(2,3) \rangle$ we will denote by $e$ and by $(a,b)$ the first and second Chern class of $E$, respectively. That is to say \begin{equation}\label{chernnotation} c_1(E)= e\, \Omega(2,4) \textrm{ and } c_2(E)= a\, \Omega(1,4) + b\, \Omega(2,3).\end{equation}  We will always assume, up to twist with a line bundle,  that $E$ is normalized, i.e. that $e=0,-1$. \\
 Let $X:=\P(E)$ be the projectivization of $E$, that is
$$
\P(E)=\Proj\left(\bigoplus_{k\geq 0}S^kE\right),
$$
 with projection $\pi: X \to \G(1,4) $. Denote by $H$ the pullback of the ample generator of $\Pic(\G(1,4))$ and by $L$ the class of the tautological line bundle $\cO(1)$ of $X$. The anticanonical bundle of $X$ is given by
$$\cO(-K_{X})=\cO(2L+ (5-e)H).$$
We will assume that $E$ is {\it Fano}, i.e. that $-K_X$ is ample. Equivalently, the $\Q$-twisted bundle $E((5-e)/2)$ is ample.

\section{Existence of sections and splitting}\label{sec:splitting}

Let $\ell$ be a line in $\G(1,4)$. By the ampleness of $-K_X$ we have that the possible splitting types of $E$ on $\ell$ are
\begin{equation}\label{splittingtypes}
(-2,2), (-1,1), (0,0) \textrm{ if }  e=0, \mbox{ and }
(-2,1), (-1,0) \textrm{ if } e=-1.
\end{equation}
In particular we obtain lower bounds for the set $\left\{k\in\Z|\,H^0(\G(1,4),E(k))\neq 0\right\}$. Later on we will make use of the following statement:


\begin{lem}\label{lemma:-2}
If $H^0(\G(1,4),E(-2))\neq 0$ then $E$ splits as $\cO(-2) \oplus \cO(2)$. The same is true if the condition is fulfilled by the restriction of $E$ to a general $\P^3$ in the cohomology class $\Omega(0,4)$.
\end{lem}

\begin{proof} Since $H^0(\ell,E(-2)_{|\ell})= 0$ for every line $\ell$ on which $E$ has splitting type different from $(-2,2)$, the existence of a non-zero global section of $E(-2)$ implies that $(-2,2)$ is the splitting type of $E$ on the general line of $\G(1,4)$. But then semicontinuity, together with (2) above, tells us that this is in fact the splitting type of $E$ on every line of the Grassmannian. Then $E$ is uniform and its splitting follows from \cite[Th\'eor\`eme 1]{G} or \cite[Theorem 4.1]{MOS}.

Note that, by (\ref{splittingtypes}) the uniformity follows if we had that the splitting type of $E$ at a general line is $(-2,2)$. Then the same proof works if we assume that $H^0(\P^3,E_{|\P^3}(-2))\neq 0$ for a general $\P^3$.
\end{proof}

Now we will translate the effectiveness of $c_2(E(j))$ into some numerical conditions which the integers $e,a,b$, defined in (\ref{chernnotation}), must satisfy:

\begin{lem}\label{lemma:posc3} Assume that for some integer $j$
\begin{equation}\label{pure}
 H^0(\G(1,4),E(j-1)) =  0 \quad\mbox{ and }\quad H^0(\G(1,4),E(j)) \not =  0.
\end{equation}
Then  $a+j(e+j) \geq 0$ and $b+j(e+j)\geq 0$, and $E\cong\cO(-j)\oplus\cO(e+j)$ if and only if $a+j(e+j) = b+j(e+j) =0$.\\
Moreover  if  condition {\rm (\ref{pure})} is fulfilled by the restriction of $E$ to a $\P^3$ in the cohomology class $\Omega(0,4)$ then $a+j(e+j) \geq 0$, and equality holds if and only if the restriction of $E$ to such $\P^3$ splits as $\cO_{\P^3}(-j)\oplus\cO_{\P^3}(e+j)$.\end{lem}

\begin{proof} Let $\sigma$ be a  section of $E(j)$ and let $Z:=\{\sigma = 0\}$ be its zero set. Since $H^0(\G(1,4),E(j-1))=0$, then $Z$ is either empty or a codimension two subvariety of $\G(1,4)$ in the cohomology class $(a+j(e+j))\Omega(1,4) + (b+j(e+j))\Omega(2,3)$. In particular $a+j(e+j), b+j(e+j)\geq 0$ and equalities hold if and only if $Z$ is empty. If  $Z$ is empty, then the cokernel $\cL_\sigma$ of $\sigma:\cO_{G(1,4)}\to E(j)$ is a line bundle, and Kodaira vanishing theorem applied to $\cL_\sigma$ tells us that we have an isomorphism $$E(j)\cong\cO_{G(1,4)}\oplus \cL_\sigma.$$ Conversely, if $E(j)$ has a direct summand $\cO_{\G(1,4)}$, then the inclusion of this subbundle into $E$ provides a section of $E(j)$ with empty zero set.\\
The statement on the restriction is proved in the same way taking into account that the cohomology class of $\P^3$ is $\Omega(0,4)$ and the vanishing of the intersection product $\Omega(2,3)\Omega(0,4)=0$.
\end{proof}

The next trivial lemma will be useful later:

\begin{lem}\label{lem:splitP3}
Consider a $\P^3$ in the cohomology class $\Omega(0,4)$ such that $E_{|\P^3}$ splits as $\cO_{\P^3}(k)\oplus\cO_{\P^3}(r)$.  The pair $(k,r)$ is completely determined by $e$ and $a$. In particular, if for every $\P^3$ in the cohomology class $\Omega(0,4)$ the restricion $E_{|\P^3}$ is a direct sum of line bundles, then $E$ is uniform.
\end{lem}

\begin{proof}
If $E_{|\P^3}\cong\cO_{\P^3}(k)\oplus\cO_{\P^3}(r)$, then $$k+r=c_1(E_{|\P^3})=c_1(E)\Omega(0,4)=e,\quad kr=c_2(E_{|\P^3})=c_2(E)\Omega(0,4)=a.$$
Then $k$ and $r$ are the only solutions of the equation $x^2-ex+a=0$, hence they are determined by $e$ and $a$.
\end{proof}

\begin{cor}\label{cor:posc3} If $H^0(\P^3,E_{|\P^3}(-1)) \not =  0$ and $H^0(\P^3,E_{|\P^3}(-2))=0$ for the general $\P^3\subset\G(1,4)$, then $a\geq e-1$ and equality holds if and only if $E$ splits as a sum of line bundles $\cO(1)\oplus\cO(e-1)$.
\end{cor}

\begin{proof} The first assertion follows directly from Lemma \ref{lemma:posc3}. Assume that $a=e-1$.
By Lemma \ref{lemma:posc3} again, it follows that the restriction of $E$ to a general $\P^3$ splits as $\cO_{\P^3}(1)\oplus\cO_{\P^3}(e-1)$. If this were the case for every $\P^3$, then $E$ would be uniform and we could conclude the splitting of $E$  (\cite[Th\'eor\`eme 1]{G}, \cite[Theorem 4.1]{MOS}). Thus we may assume that there exists a $\P^3$ for which $H^0(\P^3,E_{|\P^3}(-2))\neq 0$. Arguing as in Lemma \ref{lemma:-2}, we get that $E_{|\P^3}$ splits as $\cO_{\P^3}(2)\oplus\cO_{\P^3}(-2)$, contradicting Lemma \ref{lem:splitP3}.
\end{proof}

In order to actually get the existence of sections of a suitable twist of $E$ we will apply Le Potier vanishing Theorem as in the following

\begin{lem}\label{lemma:LP}
If $j \ge -2$ then $h^0(\G(1,4),E(j)) \ge \chi (\G(1,4),E(j))$. The same is true if $j \ge -1$ for the restriction of $E$ to  a $\P^3$ in the cohomology class $\Omega(0,4)$.
\end{lem}

\begin{proof} Notice that $E$ is Fano, hence $E(3)$ is ample and applying Le Potier vanishing Theorem \cite[II,~Thm.~7.3.5]{L} we get that
$$\chi(E(j))=h^0(\G(1,4),E(j))-h^1(\G(1,4),E(j)), \mbox{ for }j\geq -2.$$
The second part of the statement is analogous.
\end{proof}

\section{Proof of Theorem \ref{thm:main}}\label{sec:proof}

{\noindent \bf Step 1: Reduction of the set of possible Chern classes of $E$}

Let us denote $m=(5-e)/2$. By hypothesis the $\Q$-twist $E(m)$ is ample and,
in particular the restrictions of the $\Q$-bundle $E(m)$ to a $\P^3$ in the class $\Omega(0,4)$ and to a $\P^2$ in the class $\Omega(1,2)$ have positive Chern classes (see \cite{BG}), i.e. $a+me+m^2>0$ and $b+me+m^2>0$. Therefore we get $a,b\geq -6$ if
$e=0$ and $a,b>-6$ if $e=-1$.\\
By the positivity of the third Schur polynomial ($c_1^2-2c_2$, see. \cite[8.3]{L}) of $E(m)$ against the cycles $\Omega(0,4)$ and $\Omega(1,3)$ we get $a \le 6, b \le 12-a$ if $e=0$ and $a \le 6, b \le 13-a$ if $e=-1$.\\
Now (with the help of the Maple Schubert package) we use the Riemann-Roch formula to compute $\chi(E(k))$, $k\in\Z$, for all possible values left of $a$ and $b$, and we exclude those for which the result is not an integer for some $k$. This is the analogue of the Schwarzenberger's conditions on the projective space.
We are left with the following cases:

$$
\begin{array}{|c|c|c|}
\hline
& e=0 & \;e=-1\; \\\hline
(a,b) & (-4,-4)& \;(6,6) \;\\\hline
(a,b) &(-4,12)& \;(-2,-2) \;\\\hline
(a,b) &(-1,-1)& \;(-2,7) \;\\\hline
(a,b) &(-1,3)& \;(0,1) \;\\\hline
(a,b) &(0,0)& \;(0,0) \;\\\hline
\end{array}
$$\par
\smallskip
Furthermore, for $a=b=6$ the Riemann-Roch formula gives us $\chi(E(5))=-935$. On the other hand, Griffiths vanishing Theorem \cite[II,~Thm.~7.3.1]{L} provides $H^i(\G(1,4),E(5))=0$ for $i>0$, a contradiction.

\medskip

{\noindent \bf Step 2: Characterizing the case $E\cong\cO(-2)\oplus\cO(2)$}

By Lemma \ref{lemma:-2}, if $H^0(\P^3, E_{|\P^3}(-2))\neq 0$ for the general $\P^3$ in the class $\Omega(0,4)$ then $E \simeq \cO(-2) \oplus \cO(2)$, and in particular $a=b=-4$. Conversely, if $(a,b)=(-4,-4)$ then, by Riemann-Roch and Lemma \ref{lemma:LP}, we get $H^0(\G(1,4), E(-2)) > 0$ and $E \simeq \cO(-2) \oplus \cO(2)$ by Lemma \ref{lemma:-2}.

 As a consequence, we may assume, in the remaining cases, that
 \begin{equation}\label{vanishing-2} H^0(\P^3, E_{|\P^3}(-2))=0, \textrm{ for the general } \P^3 \textrm{ in }  \Omega(0,4).\end{equation}

{\noindent \bf Step 3: The case $a\neq 0$}

In this case Riemann-Roch formula for $E_{|\P^3}$ provides:

$$
\begin{array}{|c|c|}
\hline
(e,a,b) & \;\chi(E_{|\P^3}(-1))\; \\\hline
(0,-4,12)& \;4 \;\\\hline
(0,-1,-1)& \;1 \;\\\hline
(0,-1,3)& \;1 \;\\\hline
(-1,-2,-2)& \;1 \;\\\hline
(-1,-2,7)& \;1 \;\\\hline
\end{array}
$$
and, in particular, by Lemma \ref{lemma:LP}, $H^0(\P^3, E_{|\P^3}(-1))\neq 0$  for the general $\P^3$ in the cohomology class $\Omega(0,4)$. Then, using the assumption (\ref{vanishing-2}) together with Corollary \ref{cor:posc3}, we obtain that the case $(0,-4,12)$ is not possible and that $E$ splits in the rest of the cases.
It follows that the only possibilities are, either:
\begin{itemize}
\item $(e,a,b)=(0,-1,-1)$ and $E\cong\cO(-1)\oplus\cO(1)$, or
\item $(e,a,b)=(-1,-2,-2)$ and $E\cong\cO(-2)\oplus\cO(1)$.
\end{itemize}

\noindent{\bf Step 4: The case $a=0$}

Note first that in this case $H^0(\P^3,E_{|\P^3}(-2))=0$ for every $\P^3$ in $\Omega(0,4)$. In fact, if the restriction of $E(-2)$ to some $\P^3$ had sections, then, arguing as in Lemma \ref{lemma:-2}, $E_{|\P^3}$ would split as $\cO(-2)\oplus\cO(2)$ and Lemma \ref{lem:splitP3} would imply that $a=-4$.

We claim that, moreover, $H^0(\P^3,E_{|\P^3}(-1))=0$ for every $\P^3$. Assume that this is not the case for some $\P^3$ and let $Z$ be the set of zeroes of a non-zero global section $\sigma$ of $E_{|\P^3}(-1)$, which is, by the vanishing of $H^0(\P^3,E_{|\P^3}(-2))$, a curve of degree $c_2(E_{|\P^3}(-1))$.  If $e=0$, then $Z$ is a line, contradicting the adjunction formula $K_Z = (K_{\P^3} + c_1(E(-1)))_{|Z}=(\cO_{\P^3}(-6))_{|Z}.$
If else $e=-1$, let $\ell$ be a line meeting $Z$. The possible splittings of $E(-1)$ on $\ell$ are $(-3,0)$ or $(-2,-1)$ (see (\ref{splittingtypes})), so $\sigma$ cannot vanish on any point of $\ell$, a contradiction.

Finally the Riemann-Roch formula, together with Lemma \ref{lemma:LP}, tells us that $H^0(\P^3, E_{|\P^3})\neq 0$ for every $\P^3$ in $\Omega(0,4)$, hence, using Lemma  \ref{lem:splitP3}, $E_{|\P^3}$ splits as $\cO\oplus\cO$ or $\cO(-1)\oplus\cO$. In particular $E$ is uniform, necessarily of type $(0,0)$ or $(0,-1)$, and this allows us to conclude (using  \cite[Th\'eor\`eme 1]{G}, \cite[Theorem 4.1]{MOS}) that $E$ is isomorphic either to $\cO^{\oplus 2}$, or to $\cO\oplus\cO(-1)$, or to the universal bundle $\cQ$. This finishes the proof.
\qed

\bibliographystyle{amsalpha}

\end{document}